\documentclass{article}

\usepackage{graphicx}
\usepackage{verbatim}
\usepackage{enumitem}
\usepackage{caption}

\usepackage{float}
\usepackage{fancybox}
\usepackage{textcomp}
\usepackage{color}

\usepackage{amsmath}
\usepackage{amsfonts}
\usepackage{amssymb}
\usepackage{amsthm}
\usepackage{makeidx,bbm}

\usepackage{mathrsfs}
\usepackage{eurosym}
\usepackage[T1]{fontenc}

\usepackage{makeidx}

\usepackage[utf8]{inputenc}
\usepackage[english]{babel}
\usepackage{young}
\usepackage[vcentermath]{youngtab}
\usepackage{tikz}
\usetikzlibrary{calc}
\usepackage[nomessages]{fp}
\usepackage{amsthm}
\usepackage{adjustbox}
\usepackage{amsmath}
\usepackage[affil-it]{authblk}
\usepackage{mathtools}
\usepackage{chemfig}
\usepackage{parallel,enumitem}
\usepackage{pdfcolparallel}
\usepackage{setspace}
\usepackage{fancyhdr}

\theoremstyle{plain}
\newtheorem{thm}{Theorem}[section]

\newtheorem{cor}{Corollary}[thm]
\newtheorem{conj}[thm]{Conjecture}

\newtheorem{remark}[thm]{Remark}
\theoremstyle{definition}
\newtheorem{defn}[thm]{Definition}
\newtheorem{exmp}[thm]{Example}

\newcommand{\R}{\mathbb{R}}
\newcommand{\vs}{\vec w}

\evensidemargin \oddsidemargin
\begin{document} 

\title{Planar graphs and Stanley's Chromatic Functions} 
\author{Alexander Paunov} 
\date{February 2016} 
\maketitle

\pagestyle{plain}

\begin{abstract}

  This article is dedicated to the study of positivity phenomena
  for the chromatic symmetric function of a graph with respect to various bases of symmetric functions.

We give a new proof of Gasharov's theorem on the Schur-positivity of  the chromatic symmetric function of a $(3+1)$-free poset. We present a combinatorial interpretation of the Schur-coefficients in
  terms of planar networks. Compared to Gasharov's proof, it gives a
  clearer visual illustration of the cancellation procedures and is
  quite similar in spirit to the proof of monomial positivity of Schur
  functions via the Lindström–Gessel–Viennot lemma. 

  We apply a similar device to the $e$-positivity problem of chromatic
  functions. Following Stanley, we analyze certain analogs
  of symmetric functions attached to graphs instead of working with chromatic
  symmetric functions of graphs directly.  We introduce a new
  combinatorial object: the {\em correct} sequences of unit interval
  orders, and, using these, we reprove monomial positivity of
  $G$-analogues of the power sum symmetric functions.

\end{abstract}

\section{Introduction}

Let $G$ be a finite graph, $V(G)$ - the set of vertices of $G$, $E(G)$ - the set of edges of $G$. 

\begin{defn} \label{coloring} A {\em proper coloring} $c$ of  $G$ is
  a map $$c:V\rightarrow\mathbb{N}$$ such that no two adjacent
  vertices are colored in the same color.
\end{defn}

For each coloring $c$ we define a monomial $$x^c = \prod_{v\in
  V}x_{c(v)},$$ where $x_1, x_2, ..., x_n,...$ are commuting
variables.  We denote by $\Pi(G)$ the set of all proper colorings
of $G$, and by  $\Lambda$ the ring of symmetric functions in the infinite set
of variables $\{x_1, x_2,...\}.$

In \cite{Stanley95a}, Stanley defined the chromatic symmetric function of a graph.

\begin{defn} \label{chromfunction} The \em chromatic symmetric
  function \normalfont $X_G\in\Lambda$ of a graph $G$ is the sum of the monomials $x^c$
  over all proper colorings of
  $G$: $$X_G=\sum\limits_{c\in\Pi(G)}x^c.$$
\end{defn}

\begin{defn} \label{efunc} Denote by $e_m$ the $m$-th elementary
  symmetric function:
  $$e_m = \sum\limits_{i_1<i_2<...<i_m}x_{i_1}\cdot
  x_{i_2}\cdot...\cdot x_{i_m},$$
  where $i_1,..,i_k\in \mathbb{N}$.  Given a non-increasing sequence of
  positive integers (we will call these {\em partitions})
  $$\lambda = (\lambda_1\geq \lambda_2\geq...\geq\lambda_k),\ \lambda_i\in
  \mathbb{N},$$
  we define the elementary symmetric function
  $e_{\lambda} = \prod\limits_{i=1}^k e_{\lambda_i}.$ These functions
  form a basis of $\Lambda.$
\end{defn}
For a natural number $k$, we denote by $1^k$ the partition $\lambda$ of length $k$, where $$\lambda_1=\lambda_2=...=\lambda_k=1.$$

\begin{defn} \label{epos} A symmetric function $X\in \Lambda$ is \em
  $e$-positive \normalfont if it has non-negative coefficients in the
  basis of the elementary symmetric functions.
\end{defn}

\begin{defn} \label{pfunc} Denote by $p_m$ the $m$-th power sum
  symmetric function: $$p_m =
  \sum\limits_{i\in\mathbb{N}}x^m_{i}.$$ Given a
  partition $\lambda = (\lambda_1\geq \lambda_2\geq...\geq\lambda_k)$, we define the power sum
  symmetric function
  $p_{\lambda} = \prod\limits_{i=1}^k p_{\lambda_i}.$ These functions also
  form a basis of $\Lambda.$
\end{defn}

\begin{defn} \label{mfunc}
Given a partition $\lambda = (\lambda_1\geq \lambda_2\geq...\geq\lambda_k)$,  we define the monomial symmetric function $$m_\lambda=\sum\limits_{i_1<i_2<...<i_k}\sum\limits_{\lambda'\in S_k(\lambda)}x_{i_1}^{\lambda_{1}'}\cdot
  x_{i_2}^{\lambda_{1}'}\cdot...\cdot x_{i_k}^{\lambda_{k}'},$$
where the inner sum is taken over the set of all permutations of the sequence $\lambda$, denoted by $S_k(\lambda)$.
\end{defn}

\begin{exmp}
  The chromatic symmetric function of $K_n$, the complete graph on $n$
  vertices, is $e$-positive: $X_{K_n} = n!\,e_n$.
\end{exmp}

\begin{defn} \label{incgraph} For a poset $P$, the \em incomparability
  graph\normalfont, $\textnormal{inc}(P)$, is the graph with elements
  of $P$ as vertices, where two vertices are connected if and only if
  they are not comparable in $P$.
\end{defn} 

\begin{defn} \label{nplusmfree} Given a pair of natural numbers
  $a,b\in\mathbb{N}^2$, we say that a poset $P$ is \em (a+b)-free
  \normalfont if it does not contain a length-$a$ and a length-$b$
  chain, whose elements are mutually incomparable.
\end{defn} 

\begin{defn} A unit interval order (UIO) is a partially ordered set
  which is isomorphic to a finite subset of $U\subset\R$ with the following poset structure:
\[ \text{for } u,w\in U:\    u\succ w \text{ iff } u\ge w+1.
\]
Thus $u$ and $w$ are incomparable precisely when $|u-w|<1$ and we will
use the notation $u\sim w$ in this case. 
\end{defn}
\begin{thm}[Scott-Suppes \cite{Scott-Suppes54}]\label{S_S}
A finite poset $P$ is a UIO if and only if it is $(2+2)$- and $(3+1)$-free.
\end{thm}

Stanley~\cite{Stanley95a} initiated the study of incomparability
graphs of $(3+1)$-free partially ordered sets. Analyzing the chromatic
symmetric functions of these incomparability graphs,
Stanley~\cite{Stanley95a} stated the following positivity conjecture.

\begin{conj}[Stanley] \label{eposconj}
If $P$ is a $(3+1)$-free poset, then $X_{\textnormal{inc}(P)}$ is $e$-positive.
\end{conj}

For a graph $G$ let us denote by ${c_\lambda}(G)$ the coefficients of $X_G$ with respect to the $e$-basis. We omit the index $G$ whenever this causes no confusion:
$$X_G=\sum\limits_{\lambda}c_{\lambda}e_\lambda.$$

Conjecture~\ref{eposconj} has been verified with the help of computers
for up to 20-element posets~\cite{Guay-Paquet13}. In 2013,
Guay-Paquet~\cite{Guay-Paquet13} showed that to prove this conjecture,
it would be sufficient to verify it for the case of $(3+1)$- and
$(2+2)$-free posets, i.e. for unit interval orders (see Theorem~\ref{S_S}). More precisely:

\begin{thm}[Guay-Paquet]\label{G_P}
  Let $P$ be a $(3+1)$-free poset. Then, $X_\mathrm{inc}(P)$ is a
  convex combination of the chromatic symmetric
  functions $$\{X_\mathrm{inc}(P')\ |\ P'\ \mathrm{is}\ \mathrm{a}\
  (3+1)\mathrm{-}\ \mathrm{and}\ (2+2)\mathrm{-free}\ \mathrm{poset}
  \}.$$

\end{thm}

The strongest general result in this direction is that of Gasharov
\cite{Gasharov94}.

\begin{defn} \label{sfunc} For a partition $\lambda = (\lambda_1\geq
  \lambda_2\geq...\geq\lambda_k)$, define the Schur
  functions \em $s_{\lambda}=\mathrm{det}(e_{\lambda_i^*+j-i})_{i,j}$,
  \normalfont where $\lambda^*$ is the conjugate partition to
  $\lambda$. The functions $\{ s_{\lambda}\}$ form a basis of
  $\Lambda$.
\end{defn}

\begin{defn} \label{spos}
A symmetric polynomial $X$ is \em $s$-positive \normalfont if it has non-negative coefficients in the basis of Schur functions.
\end{defn}

Obviously, a product of $e$-positive functions is $e$-positive. This also holds for $s$-positive functions. Thus, the equality $e_n=s_{1^n}$ implies that $e$-positive functions are $s$-positive, and thus $s$-positivity is weaker than $e$-positivity.

\begin{thm}[Gasharov] \label{sposthm} 
If $P$ is a $(3+1)$-free poset, then $X_{\textnormal{inc}(P)}$ is $s$-positive.
\end{thm}

Gasharov proved $s$-positivity by  constructing so-called $P$-tableau and finding a one-to-one correspondence between these tableau and $s$-coefficients~\cite{Gasharov94}. However, $e$-positivity conjecture~\ref{epos} is still open. 
The strongest known result on the $e$-coefficients was obtained by Stanley in~\cite{Stanley95a}. He showed that sums of $e$-coefficients over the partitions of fixed length are non-negative:

\begin{thm}[Stanley]
For a finite graph $G$ and $j\in\mathbb{N}$, suppose $$X_G=\sum\limits_{\lambda}c_{\lambda}e_\lambda,$$
and let $\text{sink}(G,j)$ be the number of acyclic orientation of $G$ with $j$ sinks. Then
$$\text{sink}(G,j)=\sum\limits_{l(\lambda)=j}c_{\lambda}.$$
\end{thm}

\begin{remark}
By taking $j=1$, it follows from the theorem that $c_n$ is non-negative.
\end{remark}

Stanley in~\cite{Stanley95a} showed that for $n\in\mathbb{N}$ and the unit interval order $P_n=\{\frac{i}{2}\}_{i=1}^n$, the corresponding $X_{\text{inc}(P_n)}$ is $e$-positive, while $e$-positivity for the UIOs $$P_{n,k}=\bigg\{\frac{i}{k+1}\bigg\}_{i=1}^n$$ with $k>1$ has not yet been proven. It was checked for small $n$ and some $k$ (see~\cite{Stanley95a}).

In this article, we give a new proof of Gasharov's theorem, which
presents a combinatorial interpretation of the $s$-coefficients in terms
of planar networks. Compared to Gasharov's proof,
it gives a clearer visual
illustration of the cancellation procedures and resembles the proof of
monomial positivity of Schur functions using Lindström–Gessel–Viennot
Lemma~\cite{Fulmek10}. This allows us to look at the positivity
problematics from a slightly different perspective: instead of working with
the chromatic symmetric function of a graph directly, we analyze
families of $G$-symmetric functions, described in Section~\ref{Ghom}, first time proposed by Stanley in~\cite{Stanley95b}.

Next, we introduce {\em correct sequences} (abbreviated as {\em corrects}),
defined below. These play a major role in the article.

\begin{defn}
Let U be a UIO.  We will call a sequence $\vs = (w_1,\dots, w_k)$ of elements of $U$ {\em
   correct} if
 \begin{itemize}
 \item 
$w_i\not\succ w_{i+1}$ for $i=1,2,\dots,k-1$ 
\item and for
 each $j=2,\dots,k$, there exists $i<j$ such that $w_i\not\prec w_j$.
 \end{itemize}
\end{defn}
Every sequence of length 1 is correct, and sequence $(w_1,w_2)$ is
correct precisely when $w_1\sim w_2$. The second condition (supposing that the first one holds) may be reformulated as follows:
for each $j=1,\dots k$, the subset $\{w_1,\dots,w_j\}\subset U$ is connected
with respect to the graph structure~${(U,\sim)}$.
Using this notation, we prove the following theorems.
\begin{thm}\label{eposn} 
Let $X_{\text{inc}(U)}=\sum\limits_{\lambda}c_\lambda e_\lambda$ be a chromatic symmetric function of the $n$-element unit interval order $U$. Then $c_n$ is equal to the number of corrects of length $n$, in which every element of $U$ is used exactly once.
\end{thm}
\begin{cor}
Let $X_{\text{inc}(P)}=\sum\limits_{\lambda}c_\lambda e_\lambda$ be a chromatic symmetric function of  $n$-element $(3+1)$-free poset $P$, then $c_n$ is a nonnegative integer.
\end{cor}
Indeed, positivity for the general case follows from Theorem~\ref{G_P}, which presents the chromatic symmetric function of a $(3+1)$-free poset as a convex combination of the chromatic symmetric functions of unit interval orders.

Stanley~\cite{Stanley95b} and Chow~\cite{Chow95} showed the positivity of $c_n$ for $(3+1)$-free posets using combinatorial techniques, and linked $e$-coefficients with the acyclic orientations of the incomparability graphs. Nevertheless, their proofs do not give visual interpretation of the cancellation procedures. 


The article is structured as follows: in Section~\ref{Ghom}, we describe
the $G$-homomorphism introduced by Stanley in~\cite{Stanley95b}, which
is essential for our approach. The new proof of Gasharov's theorem (Theorem~\ref{sposthm}) is presented
in Section~\ref{Gashproof}. The proof of Theorem~\ref{eposn} and positivity of $G$-power sum symmetric functions is can be found in Section~3.2.

\bigskip

\textbf {Acknowledgements.}  I would like to express my deep gratitude
to my advisor, Andras Szenes, for introducing me to the subject, for
his guidance and help. I am grateful to Richard Rimanyi, Emanuele Delucchi and Bart Vandereycken for helpful discussions.

\section{Stanley's $G$-homomorphism}\label{Ghom}
For a graph $G$, Stanley \cite[p.~6]{Stanley95b} defined $G$-analogues of the standard families of symmetric functions. Let $G$ be a finite graph with vertex set $V(G)=\{v_1,...,v_n\}$ and edge set $E(G)$. We will think of the elements of $V(G)$ as commuting variables.

\begin{defn} \label{eG}
For a positive integer $i$, $1\leq i, \leq n$, we define the \em $G$-analogues \normalfont of the elementary symmetric polynomials, or \em the elementary $G$-symmetric polynomials\normalfont, as follows
 $$e_i^G =\sum\limits_{\substack{\#S=i\\
     S-\mathrm{stable}}}\prod\limits_{v\in S}v,$$ where the sum is
 taken over all $i$-element subsets $S$ of $V$, in which no two
 vertices form an edge, i.e. stable subsets. We set $e_0^G=1$, and $e_i^G=0$ for $i<0$.
\end{defn}
Note that these polynomials are not necessarily symmetric.

Let $\Lambda_G\subset\mathbb{R}[v_1,...,v_n]$ be the subring generated
by $\{e_i^G\}_{i=1}^{n}$.  The map $e_i\mapsto e_i^G$ extends to a
ring homomorphism $\phi_G: \Lambda\rightarrow\Lambda_G$, called the
{\em $G$-homomorphism}. For $f\in \Lambda$, we will use the notation $f^G$ for $\phi_G(f)$.

\begin{exmp}
Given a partition $\lambda = \lambda_1\geq \lambda_2\geq...\geq\lambda_k,\ k\in \mathbb{N},$ we have $$e_{\lambda}^G = \prod\limits_{i=1}^k e_i^G,$$
$$s_{\lambda}^G=\mathrm{det}(e_{\lambda_i^*+j-i}^G).$$
\end{exmp}

For an integer function $\alpha: V\rightarrow \mathbb{N}$ and $f^G\in\Lambda_G$, let
$$v^\alpha = \prod\limits_{v\in V}v^{\alpha(v)},$$ 
 and $[v^\alpha]f^G$ stands for the coefficient of $v^\alpha$ in the polynomial $f^G\in\Lambda_G$.

Let $G^\alpha$ denote the graph, obtained by replacing every vertex $v$ of $G$ by the complete subgraph of size $\alpha(v)$: $K_{\alpha(v)}^v$. Given vertices $u$ and $v$ of $G$, a vertex of $K_{\alpha(v)}^v$ is connected to a vertex of $K_{\alpha(u)}^u$ if and only if $u$ and $v$ form an edge in $G$.

Considering the Cauchy product \cite[ch.~4.2]{Macdonald79}, Stanley \cite[p.~6]{Stanley95b} found a connection between the $G$-analogues of symmetric functions and $X_G$. Following Stanley~\cite{Stanley95b}, we set
$$T(x,v) = \sum\limits_\lambda m_\lambda(x)e^G_\lambda(v),$$
where the sum is taken over all partitions. Then 

\begin{equation}\label{gnechrom}
[v^\alpha]T(x,v)\prod\limits_{v\in V}\alpha(v)! =X_{G^\alpha}.
\end{equation}
Using the Cauchy identity 
$$\sum\limits_\lambda s_\lambda(x)s_{\lambda^*}(y)=\sum\limits_\lambda m_\lambda(x)e_\lambda(y) = \sum\limits_\lambda e_\lambda(x)m_\lambda(y)$$
and applying the $G$-homomorphism, one obtains:
\begin{equation}\label{GCauchy}
T(x,v) = \sum\limits_\lambda m_\lambda(x)e^G_\lambda(v) = \sum\limits_\lambda s_\lambda(x)s^G_{\lambda^*}(v)=T(v,x) = \sum\limits_\lambda e_\lambda(x)m^G_\lambda(v).
\end{equation}

An immediate consequence of the formulas \eqref{gnechrom} and
\eqref{GCauchy} is the following result of Stanley:
\begin{thm}[Stanley]\label{poscrit}
 For every finite graph G
\begin{enumerate}
\item $X_{G^\alpha}$ is s-positive for every
  $\alpha:V(G)\rightarrow\mathbb{N}$ if and only if $s_\lambda^G\in
  \mathbb{N}[V(G)]$ for every partition $\lambda$.
\item $X_{G^\alpha}$ is e-positive for every $\alpha:V(G)\rightarrow\mathbb{N}$ if and only if $m_\lambda^G\in \mathbb{N}[V(G)]$ for every partition $\lambda$. 
\end{enumerate}
\end{thm}

\begin{remark}\label{c_m}
If $X_{G^\alpha}=\sum\limits_{\lambda}c^\alpha_{\lambda}e_\lambda,$ then $c_\lambda^\alpha=[v^\alpha]m^G_\lambda.$ Hence, monomial positivity of $m^G_\lambda$ is equivalent to the positivity of $c_\lambda^\alpha$ for every $\alpha$.
\end{remark}

The proofs of positivity of $G$-power sum symmetric functions and  Schur $G$-symmetric functions for the case of unit interval orders can be found in~\cite{Paunov16}.

\section{Proofs of the theorems}\label{proofs}

It follows from Theorem \ref{poscrit} that to prove that the graph $G$ is $s$-positive, it is enough to show the monomial positivity of its $G$-Schur polynomials. On the other hand, Guay-Paquet (Theorem~\ref{G_P}) showed that it is sufficient to check $s$-positivity for unit interval orders in order to prove it for the general case of $(3+1)$-free posets. Therefore, in the following paragraph~\ref{Gashproof}, we analyze the functions $s_{\lambda}^G$ for the case $G=\text{inc}(U),$ where $U$ is UIO.

\subsection{A new proof of Gasharov's theorem}\label{Gashproof}

Given unit interval order $U$, we arrange the elements of $U$ according to their order on the real line. For instance, the incomparability graph of $U_8=\{v_i=\frac{i}{2}\}_{i=1}^8$, the 1-chain graph with 8 vertices, has the following labeling:

\begin{figure}[H]
\hspace*{+0.5cm} 
\begin{tikzpicture}
	[inner sep=2mm,
	place/.style={circle,draw=blue!50,fill=blue!20,thick},
	transition/.style={rectangle,draw=black!50,fill=black!20,thick}]

	\trimbox{0cm 1.5cm 0cm 1cm}{

		\foreach \x in {1,2,...,6}{
			\draw [-,line width=3pt] (2*\x-2,0)-- ($(2*\x,0)+(2,0)$) {};
		}

		\foreach \x in {1,2,...,8}{
			\node at ( 2*\x-2,0) [draw,circle, fill=blue!20, label=above:\LARGE $v_{\x}$] {};

		}
		
}
\end{tikzpicture}
\medskip
\caption{The incomparability graph of $U_8$.}
\label{P_8_1}
\end{figure}
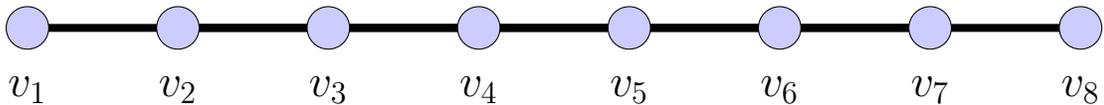

A key tool in our work is the Lindström–Gessel–Viennot
lemma~\cite{Fulmek10}.  Let $\Gamma$ be a finite directed
acyclic (i.e. without directed cycles) graph with  set of vertices $V(\Gamma)$ and set of
edges $E(\Gamma)$. Let $w:E(\Gamma)\rightarrow R$
be a weighting of the edges with values in some commutative ring $R$.  For
every directed path $\rho$, denote by $w(\rho)$ the product of the
weights of the edges in the path. Then, for every two vertices $a$ and
$b$ of $\Gamma$, let $$e(a,b)=\sum\limits_{\rho:a\rightarrow b} w(\rho),$$ where
the sum is taken over all paths from $a$ to $b$.

\begin{defn}
Let $n\in\mathbb{N},$ and let us fix two ordered
$n$-element subsets $$A=(a_1,a_2,..., a_n)\subset V(\Gamma)\
\text{and}\ B=(b_1,b_2,..., b_n)\subset V(\Gamma),$$ called {\em base} and
{\em destination} vertices, correspondingly.
  We will call a collection $\vec{\rho}=(\rho_1,...,\rho_n)$ of paths in
  $\Gamma$ a \em multipath \normalfont from $A$ to $B$ if there is a
  permutation $\sigma$ on $\{1,2,...,n\}$ such that
  $\rho_i:a_i\rightarrow b_{\sigma(i)},\ i=1,...,n.$ Given a multipath $\vec{\rho}$, we
  denote the permutation $\sigma$ by $\sigma_{\vec{\rho}}$. We call $\vec{\rho}$ \em
  non-intersecting, \normalfont if $\rho_i\cap \rho_j = \emptyset$ for
  $i\not= j.$
\end{defn}

\begin{thm}[Lindström–Gessel–Viennot]\label{L_G_V}
Let $\Gamma,$ $w:E(\Gamma)\rightarrow R$ be a weighted locally finite acyclic graph as above, $n\in\mathbb{N},$ $A=(a_1,a_2,..., a_n)\subset V(\Gamma)\
\text{and}\ B=(b_1,b_2,..., b_n)\subset V(\Gamma).$ Define  the
matrix
\[ M_{A,B}= \left( \begin{array}{cccccc}
e(a_1,b_1)      & e(a_1,b_2)  & e(a_1,b_3)& \dots  & e(a_1,b_n) \\
e(a_2,b_1)      & e(a_2,b_2)  & e(a_2,b_3)& \dots  & e(a_2,b_n) \\
\hdotsfor{5} \\
e(a_n,b_1)      & e(a_n,b_2)  & e(a_n,b_3)& \dots  & e(a_n,b_n) \\
\end{array} \right)\]
Then, the following equality holds in the ring $R$:
$$\mathrm{det}(M_{A,B}) = \sum\limits_{\substack{\vec{\rho}:A\rightarrow B\\ \text{non-int.}}} \mathrm{sign}(\sigma_{\vec{\rho}})\cdot\prod\limits_{i=1}^n w(\rho_i),$$
where the sum is taken over all non-intersecting multipaths.
\end{thm}

\begin{remark}\label{c_m}
It follows from Theorem~\ref{G_P} that to prove Gasharov's theorem, it is sufficient to verify it for unit interval orders. Here we prove Gasharov's theorem for this case. 
\end{remark}

\begin{thm}\label{sgmpos}
  Let $(U,\succ)$ be a unit interval order, $G=\mathrm{inc}(U)$ its
  incomparability graph. Then, for every partition $\lambda$,
  $s_\lambda^G \in \mathbb{N}[V(G)]$.
\end{thm}

\begin{proof}

  We prove the monomial positivity of $s^G_{\lambda^*}$ by constructing
  a special directed graph $\Gamma_G$, the grid of $G$, and applying
  the Lindström–Gessel–Viennot theorem to $\Gamma_G$.

  The vertices of $\Gamma_G$ are given by the pairs $(i,j)$, where $i
  \in {1,...,n+1}$ and $j\in \mathbb{N}.$ Then, for every $i\in
  \{1,...,n\}$, we denote by $\mathrm{next}(i)= {\mathrm{min}}~\{j|\
  v_j\succ v_i\}$. If such $v_j$ does not exist, then we define
  $\mathrm{next}(i)=n+1$. From every vertex $(i,\ j),\ j<n+1,$ we draw a
  directed edge to the vertex $(i,\ j+1)$ with the weight 1, and a
  directed edge to $(i+1,\ \mathrm{next}(j))$ with the weight $v_i$.
  Note that $\Gamma_G$ is planar if $U$ is a unit interval order.

  For instance, for the graph $U_8$, mentioned above, the grid
  $\Gamma_{U_8}$ is as follows:

\begin{figure}[H]
\hspace*{-0.5cm} 
\begin{tikzpicture}[scale=0.9]

    \coordinate (Origin)   at (0,0);

    \coordinate (Bone) at (0,2);
   \coordinate (Btwo) at (2,-2);

 	\coordinate (Afour) at (-2,14);
	\coordinate (Athree) at (0,14);
	\coordinate (Atwo) at (2,14);
	\coordinate (Aone) at (4,14);
	
	\node[draw,circle,inner sep=5pt,fill,red!60] at (Afour) {};
	\node[draw,circle,inner sep=5pt,fill,red!60] at (Athree) {};
	\node[draw,circle,inner sep=5pt,fill,red!60] at (Atwo) {};
	\node[draw,circle,inner sep=5pt,fill,red!60] at (Aone) {};
	
	\node[coordinate, pin = {45:$a_1$=(4,1)}] at (Aone) {};
	\node[coordinate, pin = {45:$a_2$=(3,1)}] at (Atwo) {};
	\node[coordinate, pin = {45:$a_3$=(2,1)}] at (Athree) {};
	\node[coordinate, pin = {45:$a_4$=(1,1)}] at (Afour) {};

	\node[coordinate, pin = {135:(1,2)}] at (-2,12) {};

 	\coordinate (Dfour) at (2,-2);
	\coordinate (Dthree) at (6,-2);
	\coordinate (Dtwo) at (10,-2);
	\coordinate (Done) at (12,-2);

	\node[draw,circle,inner sep=5pt,fill,blue!60] at (Dfour) {};
	\node[draw,circle,inner sep=5pt,fill,blue!60] at (Dthree) {};
	\node[draw,circle,inner sep=5pt,fill,blue!60] at (Dtwo) {};
	\node[draw,circle,inner sep=5pt,fill,blue!60] at (Done) {};

	\node[coordinate, pin = {225:$b_1$=(8,9)}] at (Done) {};
	\node[coordinate, pin = {225:$b_2$=(7,9)}] at (Dtwo) {};
	\node[coordinate, pin = {225:$b_3$=(5,9)}] at (Dthree) {};
	\node[coordinate, pin = {225:$b_4$=(3,9)}] at (Dfour) {};

	\coordinate (Shift1) at (2,-2);
	\coordinate (Shift2) at (2,-4);

    \draw[style=help lines,dashed] (-2,-2) grid[step=2cm] (14,14);

    \foreach \x in {-1,0,...,5}{
      \foreach \y in {-1,0,...,7}{
        \node[draw,circle,inner sep=2pt,fill,] at (2*\x,2*\y) {};

      }
    }

    \foreach \x in {-1,-0,...,6}{
      \foreach \y in {0,1,...,7}{
        \node[draw,circle,inner sep=2pt,fill,] at (2*\x,2*\y) {};
	 \draw [-stealth,black!80,line width=2pt] (2*\x,2*\y)
        -- ($(2*\x,2*\y)+(0,-2)$)
	node[pos=0.4,left,black!, outer sep=-3pt] {\LARGE 1};
	
      }
    }


	\foreach \x in {-1,0,...,6}{
	
		\FPeval{\z}{clip(8-0)};
		 \draw [->>,black!80,line width=2pt] (2*\x,0) -- ($(2*\x,0)+(Shift1)$) node [black,midway,fill=white] {\LARGE $v_{\z}$};
	}

\foreach \x in {-1,0,...,6}{
	\foreach \y in {1,...,7}{
		\node[draw,circle,inner sep=2pt,fill,] at (2*\x,2*\y) {};
			\FPeval{\z}{clip(8-\y)};
			\draw [->>,black!80,line width=2pt] (2*\x,2*\y) -- ($(2*\x,2*\y)+(Shift2)$) node [black,midway,fill=white] {\LARGE $v_{\z}$ };
	}
}

\end{tikzpicture}
\caption{The grid $\Gamma_{U_8}.$}
\label{smaingrid}

\end{figure}
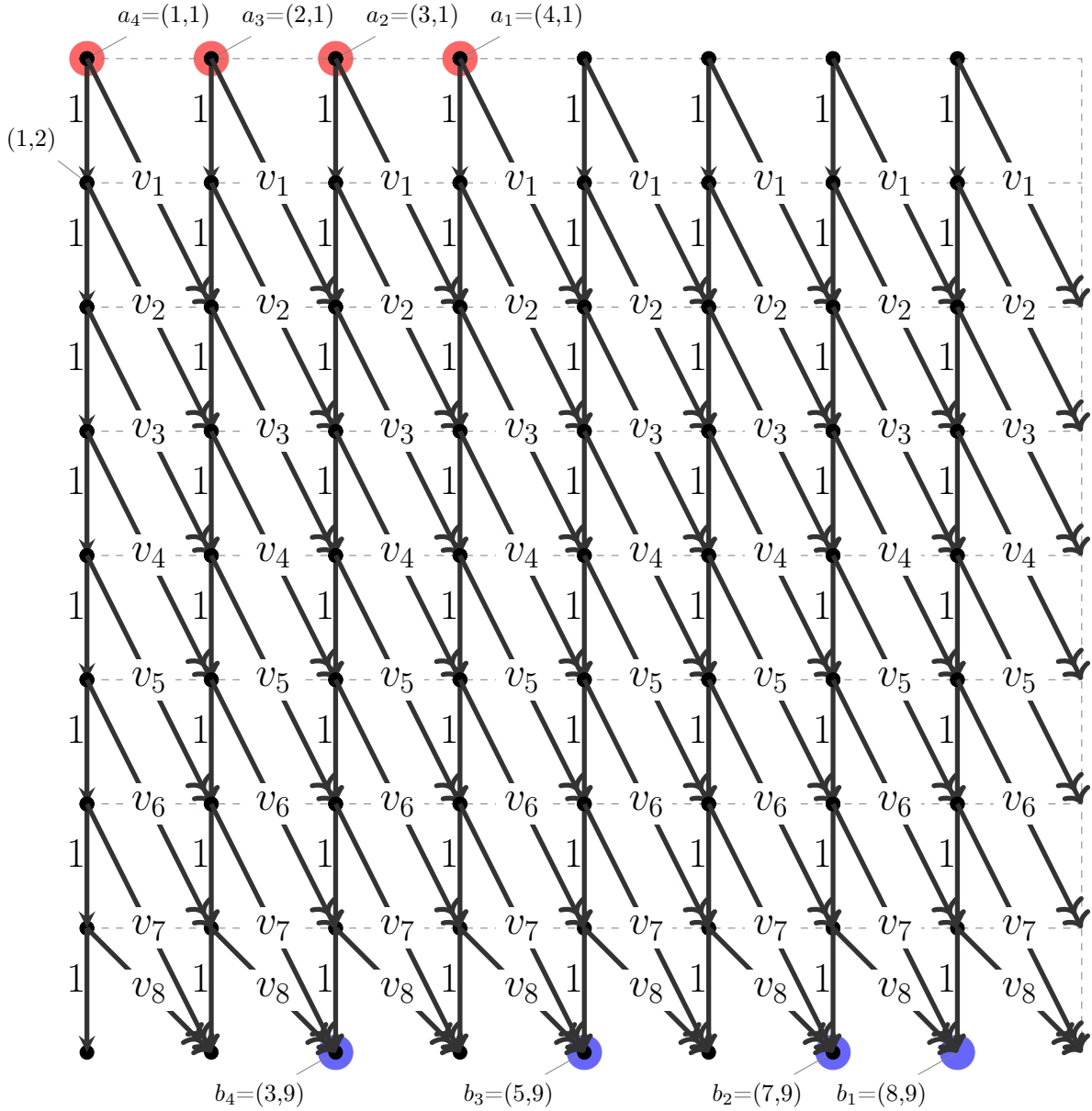

Here, the base vertices $A$, $(1,1), (2,1), (3,1),$ and $(4,1)$, are
on the top, and are colored in red. The destination vertices $B$,
$(3,9), (5,9), (7,9),$ and $(8,9)$, at the bottom, and are colored in
blue.

It easily follows from the definition of $\Gamma_G$ that, for positive
integers $i$ and $j$,  we have $$e( (i,1),\ (i+j,n+1) ) = e^G_j.$$

Note that we use the notation $e(a,b)$ for the sum over all weights of
paths from vertex $a$ to vertex $b$, and we use a similar notation
$e^G_{\lambda}$ for the $G$-elementary symmetric functions. This is
not a coincidence: for the graphs  we will consider in this
article and to which we apply Theorem~\ref{L_G_V}, $e(a,b)$ will turn out to be
the elementary $G$-symmetric function.

Now, let $k\in\mathbb{N}$, and fix a partition $\lambda_1\geq
\lambda_2\geq... \geq \lambda_k$.  Let $$A=\{a_1=(k,1),\
a_2=(k-1,1),..,a_i=(k+1-i,1),..,\ a_k =(1,1)\},$$
and $$B=\{b_1=(k+\lambda_1,n+1),\ b_2=(k-1+\lambda_2,n+1),...,\ b_i
=(k+1-i+\lambda_i,n+1)\,...,b_k=(\lambda_k+1,n+1)\}.$$ Then we have $$e( a_i,\
b_j) = e^G_{\lambda_i+j-i}.$$
The example of  $\lambda=(4,4,3,2)$ is shown on  Figure~\ref{smaingrid}.

Next, applying Theorem~\ref{L_G_V}, we obtain 
\begin{equation}\label{sposproof}
\mathrm{det}(e^G_{\lambda_i+j-i} ) = \mathrm{det}(e( a_i,\ b_j)) = \sum\limits_{\vec{\rho}:A\rightarrow B} \mathrm{sign}(\sigma_{\vec{\rho}})\prod\limits_{i=1}^n w(\rho_i),
\end{equation}
where the sum is taken over all non-intersecting multipaths from $A$ to $B$. 
The permutation $\sigma$ must be the identity permutation for all
possible non-intersecting multipaths $\vec{\rho}$ since the grid $\Gamma_G$ is planar. 
Thus, by the definition of $s_{\lambda}^G$, we have

\begin{equation}\label{sposresult}
s^G_{\lambda^*}=\mathrm{det}(e^G_{\lambda_i+j-i} ) = \sum\limits_{\substack{\vec{\rho}:A\rightarrow B \\ \text{non-intersecting}}}\prod\limits_{i=1}^n w(\rho_i),
\end{equation}
where the sum is taken over all non-intersecting multipaths from $A$ to $B$. This proves the monomial positivity of $s_{\lambda^*}^G.$
\end{proof}

\subsection{Monomial positivity of the $G$-power sum  functions.}


Let us repeat the definition of a central notion for our work, that of
{\em correct} sequences of elements of a unit interval order.
\begin{defn}
 Let $(U,\prec)$ be a unit interval order, and $G=\text{inc}(U)$. We will call a sequence $\vec{w} = (w_1,\dots, w_k)$ of elements of $U$ {\em
   correct} if
 \begin{itemize}
 \item 
$w_i\not\succ w_{i+1}$ for $i=1,2,\dots,k-1$ 
\item and for
 each $j=2,\dots,k$, there exists $i<j$ such that $w_i\not\prec w_j$.
 \end{itemize}
\end{defn}
 
We denote by $P^U_k$ the set of all correct sequences (abbreviated as  {\em
   corrects}) of length $k$. Since $G$ is uniqely defined by $U$, and we are working only with UIO, here and below we use the $U$-index instead of $G$. The $U$-analogues of symmetric functions will be analyzed.

\begin{thm}\label{Ppos}
Let $U$ be a unit interval order and $p_k^U$ the Stanley power-sum function of the corresponding incomparability graph.
Then, for
every natural $k$, we have $$p_k^U=\sum\limits_{\vec{w}\in P^U_k} w_1\cdot...\cdot w_k\ \in N[U],$$ where the sum is taken over all corrects of length $k$.
\end{thm}

\begin{proof}
To prove this theorem we express the power sum $U$-symmetric function $p_k^U$ in terms of the elementary $U$-symmetric polynomials using the determinant formula:

\begin{align}
  p^U_k = \det\begin{vmatrix}
    e^U_1    & 1         & 0      & \cdots          \\
    2e^U_2   & e^U_1       & 1      & 0      & \cdots \\
    3e^U_3   & e^U_2       & e^U_1    & 1               \\
    \vdots &           &        & \ddots & \ddots \\
    ke^U_k   & e^U_{k - 1} & \cdots &        & e^U_1
  \end{vmatrix}.
\end{align}

Note that this determinant is similar to the expression for $s^U_{(1^k)^*}$ in terms of the $e$-basis; only the first column is different:

 \begin{align}
  s^U_{(1^k)^*} = \det\begin{vmatrix}
    e^U_1    & 1         & 0      & \cdots          \\
    e^U_2   & e^U_1       & 1      & 0      & \cdots \\
    e^U_3   & e^U_2       & e^U_1    & 1               \\
    \vdots &           &        & \ddots & \ddots \\
    e^U_k   & e^U_{k - 1} & \cdots &        & e^U_1
  \end{vmatrix}.
\end{align}

Next, we take a partition $\lambda = 1^k$, a grid $\Gamma_U$, and vertices (see Theorem~\ref{L_G_V} and its use in Section~\ref{Gashproof}) $$A=\{a_1=(k,1),\
a_2=(k-1,1),..,a_i=(k+1-i,1),..,\ a_k =(1,1)\},$$
and $$B=\{b_1=(k+1,n+1),\ b_2=(k,n+1),...,\ b_i
=(k+1-i+1,n+1)\,...,b_k=(2,n+1)\},$$ corresponding to the partition $\lambda$ on the grid. For instance, the grid for $U_5$ is as follows:

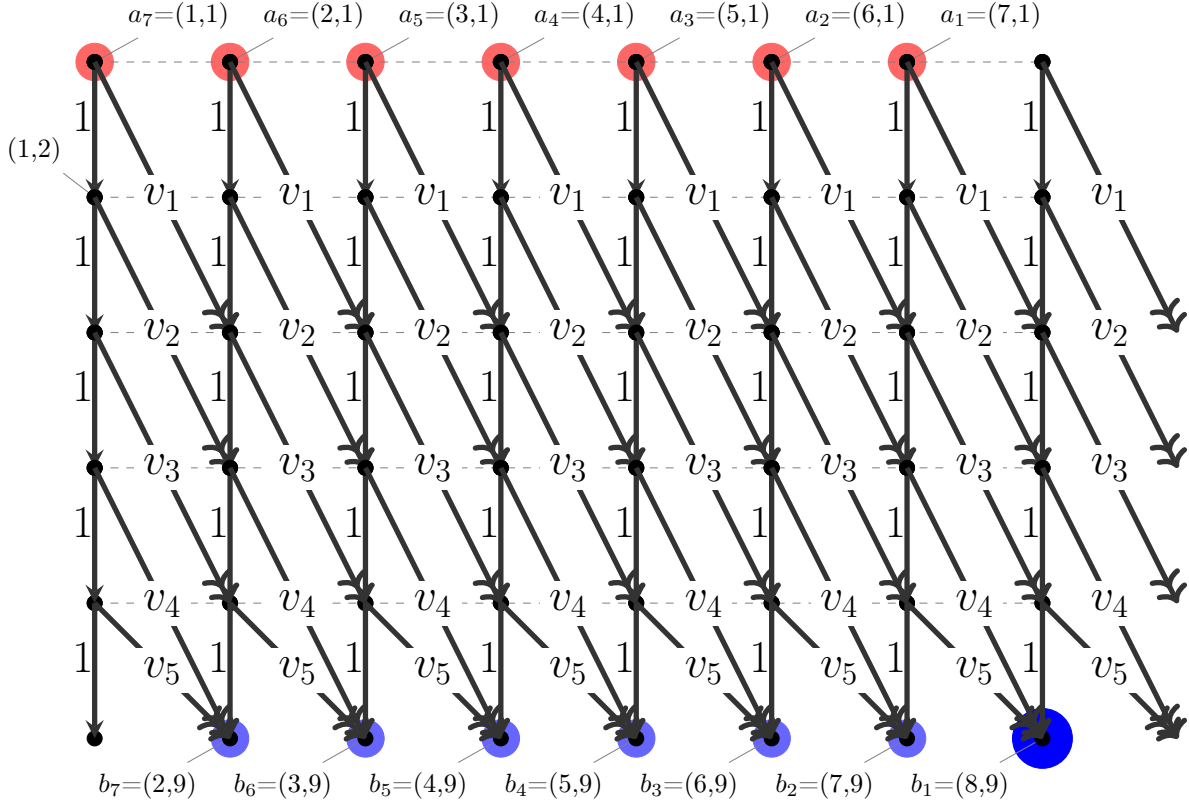
\begin{figure}[H]
\hspace*{-0.5cm} 
\begin{tikzpicture}[scale=0.9]

    \coordinate (Origin)   at (0,0);

 	\coordinate (A7) at (-2,8);
	\coordinate (A6) at (0,8);
	\coordinate (A5) at (2,8);
	\coordinate (A4) at (4,8);
 	\coordinate (A3) at (6,8);
	\coordinate (A2) at (8,8);
	\coordinate (A1) at (10,8);

	\node[draw,circle,inner sep=5pt,fill,red!60] at (A7) {};
	\node[draw,circle,inner sep=5pt,fill,red!60] at (A6) {};
	\node[draw,circle,inner sep=5pt,fill,red!60] at (A5) {};
	\node[draw,circle,inner sep=5pt,fill,red!60] at (A4) {};
	\node[draw,circle,inner sep=5pt,fill,red!60] at (A3) {};
	\node[draw,circle,inner sep=5pt,fill,red!60] at (A2) {};
	\node[draw,circle,inner sep=5pt,fill,red!60] at (A1) {};
	
	\node[coordinate, pin = {45:$a_1$=(7,1)}] at (A1) {};
	\node[coordinate, pin = {45:$a_2$=(6,1)}] at (A2) {};
	\node[coordinate, pin = {45:$a_3$=(5,1)}] at (A3) {};
	\node[coordinate, pin = {45:$a_4$=(4,1)}] at (A4) {};
	\node[coordinate, pin = {45:$a_5$=(3,1)}] at (A5) {};
	\node[coordinate, pin = {45:$a_6$=(2,1)}] at (A6) {};
	\node[coordinate, pin = {45:$a_7$=(1,1)}] at (A7) {};

	\node[coordinate, pin = {135:(1,2)}] at (-2,6) {};
        
	\coordinate (B7) at (0,-2);
	\coordinate (B6) at (2,-2);
	\coordinate (B5) at (4,-2);
 	\coordinate (B4) at (6,-2);
	\coordinate (B3) at (8,-2);
	\coordinate (B2) at (10,-2);
	\coordinate (B1) at (12,-2);

	\node[draw,circle,inner sep=5pt,fill,blue!60] at (B7) {};
	\node[draw,circle,inner sep=5pt,fill,blue!60] at (B6) {};
	\node[draw,circle,inner sep=5pt,fill,blue!60] at (B5) {};
	\node[draw,circle,inner sep=5pt,fill,blue!60] at (B4) {};
	\node[draw,circle,inner sep=5pt,fill,blue!60] at (B3) {};
	\node[draw,circle,inner sep=5pt,fill,blue!60] at (B2) {};
	\node[draw,circle,inner sep=8pt,fill,blue] at (B1) {};

	\node[coordinate, pin = {225:$b_1$=(8,9)}] at (B1) {};
	\node[coordinate, pin = {225:$b_2$=(7,9)}] at (B2) {};
	\node[coordinate, pin = {225:$b_3$=(6,9)}] at (B3) {};
	\node[coordinate, pin = {225:$b_4$=(5,9)}] at (B4) {};
	\node[coordinate, pin = {225:$b_5$=(4,9)}] at (B5) {};
	\node[coordinate, pin = {225:$b_6$=(3,9)}] at (B6) {};
	\node[coordinate, pin = {225:$b_7$=(2,9)}] at (B7) {};

	\coordinate (Shift1) at (2,-2);
	\coordinate (Shift2) at (2,-4);

    \draw[style=help lines,dashed] (-2,-2) grid[step=2cm] (12,8);

    \foreach \x in {-1,0,...,6}{
      \foreach \y in {-1,0,...,4}{
        \node[draw,circle,inner sep=2pt,fill,] at (2*\x,2*\y) {};

      }
    }

    \foreach \x in {-1,-0,...,6}{
      \foreach \y in {0,1,...,4}{
        \node[draw,circle,inner sep=2pt,fill,] at (2*\x,2*\y) {};
	 \draw [-stealth,black!80,line width=2pt] (2*\x,2*\y)
        -- ($(2*\x,2*\y)+(0,-2)$)
	node[pos=0.4,left,black, outer sep=-3pt] {\LARGE 1};
	
      }
    }


	\foreach \x in {-1,0,...,6}{
	
		\FPeval{\z}{clip(5-0)};
		 \draw [->>,black!80,line width=2pt] (2*\x,0) -- ($(2*\x,0)+(Shift1)$) node [black,midway,fill=white] {\LARGE $v_{\z}$};
	}

\foreach \x in {-1,0,...,6}{
	\foreach \y in {1,...,4}{
		\node[draw,circle,inner sep=2pt,fill,] at (2*\x,2*\y) {};
			\FPeval{\z}{clip(5-\y)};
			\draw [->>,black!80,line width=2pt] (2*\x,2*\y) -- ($(2*\x,2*\y)+(Shift2)$) node [black,midway,fill=white] {\LARGE $v_{\z}$ };
	}
}

\end{tikzpicture}
\caption{The grid $\Gamma_{U_5}$ and vertices located with respect to partition $\lambda=1^5$.}
\label{1pgrid}
\end{figure}

As in the section \ref{Gashproof}, we have $$e( a_i,\
b_j) = e^U_{1+j-i}.$$

Recall that (see Theorem~\ref{L_G_V} for more details) for
every directed path $\rho$ on $\Gamma_U$, $w(\rho)$ denotes the product of the
weights of the edges in the path. Denote by $w(\vec{\rho}\ )=(w(\rho_1),...,w(\rho_k))$ the vector of weight products over the paths of $\vec{\rho}$.

\begin{align}
 s_k^U   & = \mathrm{det}(e^U_{1+j-i} )=\mathrm{det}(e( a_i,\ b_j)) = \\
              & = \sum\limits_{\vec{\rho}=(\rho_1,...,\rho_k):A\rightarrow B} \mathrm{sign}(\sigma_{\vec{\rho}})\prod\limits_{i=1}^k w(\rho_i)= \sum\limits_{ \substack{(\rho_1,...,\rho_k):A\rightarrow B \\ \text{non-intersecting} }}\prod\limits_{i=1}^k w(\rho_i).
\end{align}

To  obtain $p^U_k$, we adjust the first column, multiplying every element by the number of its row:

\begin{align}
p^U_k & = \det\begin{vmatrix}
    e^U_1    & 1         & 0      & \cdots          \\
    2e^U_2   & e^U_1       & 1      & 0      & \cdots \\
    3e^U_3   & e^U_2       & e^U_1    & 1               \\
    \vdots &           &        & \ddots & \ddots \\
    ke^U_k   & e^U_{k - 1} & \cdots &        & e^U_1
  \end{vmatrix}=\det\begin{vmatrix}
    \ e( a_1,\ b_1)    & 1         & 0      & \cdots          \\
    2e( a_2,\ b_1)  & e( a_2,\ b_2)        & 1      & 0      & \cdots \\
    3e( a_3,\ b_1)   & e( a_3,\ b_2)       & e( a_3,\ b_3)    & 1               \\
    \vdots &           &        & \ddots & \ddots \\
    ke( a_k,\ b_1)   & e( a_k,\ b_2) & \cdots &        & e( a_k,\ b_k) \end{vmatrix}=\\&\label{mainPG} =\sum\limits_{\vec{\rho}=(\rho_1,...,\rho_k):A\rightarrow B}\mathrm{sign}(\sigma_{\vec{\rho}})\cdot\sigma_{\vec{\rho}}^{-1}(1)\prod\limits_{i=1}^k w(\rho_i).
\end{align}

In this sum every multipath has a multiplier equal to the index of the vertex from $A$, from which the corresponding path goes to $b_1$. We mark the vertex $b_1$ with a larger dot on the grid (Picture \ref{1pgrid}) to emphasize this. We cannot apply Theorem~\ref{L_G_V} here, as we did for $s^U$ functions, to obtain positive sum.

\bigskip
We will use the following notations:
\begin{itemize}
\item If we have a path $\rho$ on $\Gamma_U$, which goes from $a$ to $b$
through $z$, then let us denote by $\rho|^z$ the part of $\rho$ from
$a$ to $z$, and by $\rho|_z$ - the part of $\rho$ from $z$ to $b$.

\item If the end of the path $\rho$ coincides with the starting point
$\pi$, then we will write $\rho*\pi$ for the concatenation
of the two paths

\item For a pair of paths $(\rho,\pi)$, crossing in point $z$, we define the usual switch operation $$\text{switch}_z(\rho,\pi) = (\rho|^z*\pi|_z,\pi|^z*\rho|_z ).$$

\item Given a multipath $\vec{\rho}$ with its paths $\rho$ and $\pi$ intersecting in point $z$, we define a multipath $\delta_z(\vec{\rho})$ by replacing $(\rho,\pi)$ by $\text{switch}_z(\rho,\pi)$ in $\vec{\rho}$. Note that our map is defined correctly, because here we consider only multipaths for the partition $\lambda=1^k$: it is obvious that 3 paths of $\vec{\rho}$ cannot intersect in one point. 
Note that $$sign(\sigma({\vec{\rho}}\ ))=-sign(\sigma({\delta_{z}(\vec{\rho}\ )})).$$

\item Given an intersecting multipath $\vec{\rho}$, we denote by $z(\vec{\rho}\ )$ (or just $z$, if it is clear which multipath is considered) its intersection point with minimum absciss and maximum ordinate, i.e. the leftmost lowest intersection point. 

\end{itemize}

Next, we classify the set of multipaths in order to simplify the sum (\ref{mainPG}).
Every path $\rho$ can be uniquely defined by its weight, $w(\rho)$, which is a product over an increasing sequence (with respect to the relation $\succ$) of elements of $U$. Here, it is important to mention that incomparable elements of $U$ can not be present in a weight of any path. Hence, every multipath $\vec{\rho}=(\rho_1,...,\rho_k)$ is in one to one correspondence with its weight vector $w(\vec{\rho}\ )=(w(\rho_1),...,w(\rho_k))$.
Below, we will use the bar notation for the sets of multipaths. The corresponding sets of weight vectors will be defined using the same letters without bars. 
\begin{itemize}

\item Let $\overline{\Omega}_k$ be the set of all multipaths $\vec{\rho}=(\rho_1,...,\rho_k)$ from $A$ to $B$.

\item Let $\overline{I}_k$ be the set of all intersecting multipaths  $\vec{\rho}\in\overline{\Omega}_k$, such that the two paths from $\vec{\rho}$,
  crossing at $z(\vec{\rho}\ )$ do not end at $b_1$.

\item We denote by $\overline{P}_k$ the set of multipaths $\vec{\rho}\in\overline{\Omega}_k$, such that $w(\vec{\rho}\ )$ is correct: $$\overline{P}_k=\{\vec{\rho}\in \overline{\Omega}_k|\ w(\vec{\rho})\in P^U_k\}.$$

Note that if $\vec{\rho}\in\overline{P}_k$, then $\vec{\rho}$ is a non-intersecting multipath, since by the definition of correct $w(\vec{\rho})$ must be a tuple with non-decreasing elements (weights) with respect to the relation~$\prec$. Hence, $$\overline{I}_k\cap\overline{P}_k=\emptyset.$$

\end{itemize}

As a consequence, the sum (\ref{mainPG}) can be rewritten as 
\begin{align}
\sum\limits_{\vec{\rho}\in\overline{\Omega}_k}\mathrm{sign}(\sigma_{\vec{\rho}})\cdot\sigma_{\vec{\rho}}^{-1}(1)\prod\limits_{i=1}^k w(\rho_i)&= \sum\limits_{\vec{\rho}\in\overline{P}_k}\mathrm{sign}(\sigma_{\vec{\rho}})\cdot\sigma_{\vec{\rho}}^{-1}(1)\prod\limits_{i=1}^k w(\rho_i)\label{sumW}\\ &+\sum\limits_{\vec{\rho}\in \overline{I}_k}\mathrm{sign}(\sigma_{\vec{\rho}})\cdot\sigma_{\vec{\rho}}^{-1}(1)\prod\limits_{i=1}^k w(\rho_i)\label{sumI} \\ &+\sum\limits_{\vec{\rho}\in(\overline{\Omega}_k\setminus \overline{I}_k)\setminus\overline{P}_k}\mathrm{sign}(\sigma_{\vec{\rho}})\cdot\sigma_{\vec{\rho}}^{-1}(1)\prod\limits_{i=1}^k w(\rho_i).\label{sumJ}   
\end{align}

Let $\vec{\rho}\in \overline{I}_k$, then it is easy to see that $\delta_{z}(\vec{\rho}\ )\in \overline{I}_k, \text{ and }  \delta_{z}(\delta_{z}(\vec{\rho}\ ))=\vec{\rho}$. Hence, $\delta_{z}$ is a sign-reversing involution on $\overline{I}_k$:
$$\text{sign}(\sigma({\vec{\rho}\ }))=-\text{sign}(\sigma({\delta_{z}(\vec{\rho}\ )}))\text{ and }\delta_{z}(\overline{I}_k)=\overline{I}_k$$
On the other hand, $\delta_{z}$ does not change the multiplier:
 $$\sigma^{-1}({\delta_{z}(\vec{\rho}\ )})(1)=\sigma^{-1}_{\vec{\rho}}(1).$$
Hence, the term (\ref{sumI}) vanishes, and we have:
$$\sum\limits_{\vec{\rho}\in \overline{I}_k}\mathrm{sign}(\sigma_{\vec{\rho}})\cdot\sigma_{\vec{\rho}}^{-1}(1)\prod\limits_{i=1}^k w(\rho_i)=\delta_{z}\Bigg(\sum\limits_{\vec{\rho}\in \overline{I}_k}\mathrm{sign}(\sigma_{\vec{\rho}})\cdot\sigma_{\vec{\rho}}^{-1}(1)\prod\limits_{i=1}^k w(\rho_i)\Bigg)=-\sum\limits_{\vec{\rho}\in \overline{I}_k}\mathrm{sign}(\sigma_{\vec{\rho}})\cdot\sigma_{\vec{\rho}}^{-1}(1)\prod\limits_{i=1}^k w(\rho_i)=0.$$
Pictures \ref{Pmultipath1} and \ref{Pmultipath1switched} below illustrate this cancellation. Since neither of the 2 paths intersecting at $z$ end at $b_1$, a switch at $z$ changes the sign, but not the multiplier, and the contributions of $\vec{\rho}$ and $\delta_{z}(\vec{\rho}\ )$ cancel:
\newpage
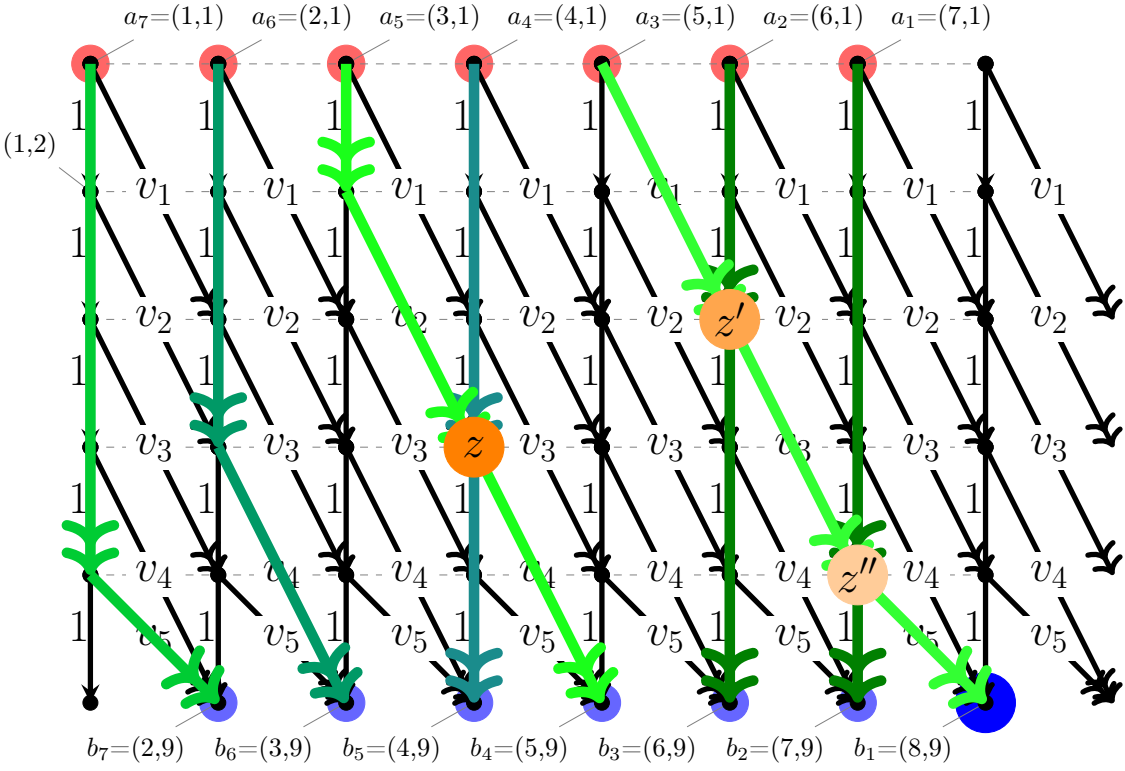
\begin{figure}[H]
\hspace*{-0.5cm} 
\begin{tikzpicture}[scale=0.85]

    \coordinate (Origin)   at (0,0);

 	\coordinate (A7) at (-2,8);
	\coordinate (A6) at (0,8);
	\coordinate (A5) at (2,8);
	\coordinate (A4) at (4,8);
 	\coordinate (A3) at (6,8);
	\coordinate (A2) at (8,8);
	\coordinate (A1) at (10,8);

	\node[draw,circle,inner sep=5pt,fill,red!60] at (A7) {};
	\node[draw,circle,inner sep=5pt,fill,red!60] at (A6) {};
	\node[draw,circle,inner sep=5pt,fill,red!60] at (A5) {};
	\node[draw,circle,inner sep=5pt,fill,red!60] at (A4) {};
	\node[draw,circle,inner sep=5pt,fill,red!60] at (A3) {};
	\node[draw,circle,inner sep=5pt,fill,red!60] at (A2) {};
	\node[draw,circle,inner sep=5pt,fill,red!60] at (A1) {};
	
	\node[coordinate, pin = {45:$a_1$=(7,1)}] at (A1) {};
	\node[coordinate, pin = {45:$a_2$=(6,1)}] at (A2) {};
	\node[coordinate, pin = {45:$a_3$=(5,1)}] at (A3) {};
	\node[coordinate, pin = {45:$a_4$=(4,1)}] at (A4) {};
	\node[coordinate, pin = {45:$a_5$=(3,1)}] at (A5) {};
	\node[coordinate, pin = {45:$a_6$=(2,1)}] at (A6) {};
	\node[coordinate, pin = {45:$a_7$=(1,1)}] at (A7) {};

	\node[coordinate, pin = {135:(1,2)}] at (-2,6) {};
        
	\coordinate (B7) at (0,-2);
	\coordinate (B6) at (2,-2);
	\coordinate (B5) at (4,-2);
 	\coordinate (B4) at (6,-2);
	\coordinate (B3) at (8,-2);
	\coordinate (B2) at (10,-2);
	\coordinate (B1) at (12,-2);

	\node[draw,circle,inner sep=5pt,fill,blue!60] at (B7) {};
	\node[draw,circle,inner sep=5pt,fill,blue!60] at (B6) {};
	\node[draw,circle,inner sep=5pt,fill,blue!60] at (B5) {};
	\node[draw,circle,inner sep=5pt,fill,blue!60] at (B4) {};
	\node[draw,circle,inner sep=5pt,fill,blue!60] at (B3) {};
	\node[draw,circle,inner sep=5pt,fill,blue!60] at (B2) {};
	\node[draw,circle,inner sep=8pt,fill,blue] at (B1) {};

	\node[coordinate, pin = {225:$b_1$=(8,9)}] at (B1) {};
	\node[coordinate, pin = {225:$b_2$=(7,9)}] at (B2) {};
	\node[coordinate, pin = {225:$b_3$=(6,9)}] at (B3) {};
	\node[coordinate, pin = {225:$b_4$=(5,9)}] at (B4) {};
	\node[coordinate, pin = {225:$b_5$=(4,9)}] at (B5) {};
	\node[coordinate, pin = {225:$b_6$=(3,9)}] at (B6) {};
	\node[coordinate, pin = {225:$b_7$=(2,9)}] at (B7) {};

	\coordinate (Shift1) at (2,-2);
	\coordinate (Shift2) at (2,-4);

    \draw[style=help lines,dashed] (-2,-2) grid[step=2cm] (12,8);

    \foreach \x in {-1,0,...,6}{
      \foreach \y in {-1,0,...,4}{
        \node[draw,circle,inner sep=2pt,fill,] at (2*\x,2*\y) {};

      }
    }

    \foreach \x in {-1,-0,...,6}{
      \foreach \y in {0,1,...,4}{
        \node[draw,circle,inner sep=2pt,fill,] at (2*\x,2*\y) {};
	 \draw [-stealth,black,line width=2pt] (2*\x,2*\y)
        -- ($(2*\x,2*\y)+(0,-2)$)
	node[pos=0.4,left,black, outer sep=-3pt] {\LARGE 1};
	
      }
    }


	\foreach \x in {-1,0,...,6}{
	
		\FPeval{\z}{clip(5-0)};
		 \draw [->>,black,line width=2pt] (2*\x,0) -- ($(2*\x,0)+(Shift1)$) node [black,midway,fill=white] {\LARGE $v_{\z}$};
	}

\foreach \x in {-1,0,...,6}{
	\foreach \y in {1,...,4}{
		\node[draw,circle,inner sep=2pt,fill,] at (2*\x,2*\y) {};
			\FPeval{\z}{clip(5-\y)};
			\draw [->>,black,line width=2pt] (2*\x,2*\y) -- ($(2*\x,2*\y)+(Shift2)$) node [black,midway,fill=white] {\LARGE $v_{\z}$ };
	}
}


\draw [->>,black!50!green,line width=4pt] (A1) -- (10,0);
\draw [->>,black!50!green,line width=4pt] (10,0) -- (B2);

\draw [->>,black!50!green,line width=4pt] (A2) -- (8,4);
\draw [->>,black!50!green,line width=4pt] (8,4) -- (B3);

\draw [->>,green!80,line width=4pt] (A3) -- (8,4);
\draw [->>,green!80,line width=4pt] (8,4) -- (10,0);
\draw [->>,green!80,line width=4pt] (10,0) -- (B1);

\draw [->>,blue!50!green!90,line width=4pt] (A4) -- (4,2);
\draw [->>,blue!50!green!90,line width=4pt] (4,2) -- (B5);

\draw [->>,green!90,line width=4pt] (A5) -- (2,6);
\draw [->>,green!90,line width=4pt] (2,6) -- (4,2);
\draw [->>,green!90,line width=4pt] (4,2) -- (B4);

\draw [->>,blue!40!green,line width=4pt] (A6) -- (0,2);
\draw [->>,blue!40!green,line width=4pt] (0,2) -- (B6);

\draw [->>,blue!20!green,line width=4pt] (A7) -- (-2,0);
\draw [->>,blue!20!green,line width=4pt] (-2,0) -- (B7);


\node[ draw,circle,inner sep=8pt,fill,orange] at (4,2) {};
\node[] at (4,2) {\LARGE{$z$}};

\node[ draw,circle,inner sep=8pt,fill,orange!70] at (8,4) {};
\node[] at (8,4) {\LARGE{$z'$}};

\node[ draw,circle,inner sep=8pt,fill,orange!40] at (10,0) {};
\node[] at (10,0) {\LARGE{$z''$}};

\end{tikzpicture}
\caption{The grid $\Gamma_{U_5}$ and multipath $\vec{\rho}$.}
\label{Pmultipath1}

\end{figure}
The path $\rho_5$: $a_5\to b_4$ intersects the path
$\rho_4$: $a_4\to b_5$ at the  point $z$. After the switch at $z$, we have the paths $\rho'_5$: $a_5\to b_5$ and $\rho'_4$: $a_4\to b_4$:
\begin{figure}[H]
\hspace*{-0.5cm} 
\begin{tikzpicture}[scale=0.85]

    \coordinate (Origin)   at (0,0);

 	\coordinate (A7) at (-2,8);
	\coordinate (A6) at (0,8);
	\coordinate (A5) at (2,8);
	\coordinate (A4) at (4,8);
 	\coordinate (A3) at (6,8);
	\coordinate (A2) at (8,8);
	\coordinate (A1) at (10,8);

	\node[draw,circle,inner sep=5pt,fill,red!60] at (A7) {};
	\node[draw,circle,inner sep=5pt,fill,red!60] at (A6) {};
	\node[draw,circle,inner sep=5pt,fill,red!60] at (A5) {};
	\node[draw,circle,inner sep=5pt,fill,red!60] at (A4) {};
	\node[draw,circle,inner sep=5pt,fill,red!60] at (A3) {};
	\node[draw,circle,inner sep=5pt,fill,red!60] at (A2) {};
	\node[draw,circle,inner sep=5pt,fill,red!60] at (A1) {};
	
	\node[coordinate, pin = {45:$a_1$=(7,1)}] at (A1) {};
	\node[coordinate, pin = {45:$a_2$=(6,1)}] at (A2) {};
	\node[coordinate, pin = {45:$a_3$=(5,1)}] at (A3) {};
	\node[coordinate, pin = {45:$a_4$=(4,1)}] at (A4) {};
	\node[coordinate, pin = {45:$a_5$=(3,1)}] at (A5) {};
	\node[coordinate, pin = {45:$a_6$=(2,1)}] at (A6) {};
	\node[coordinate, pin = {45:$a_7$=(1,1)}] at (A7) {};

	\node[coordinate, pin = {135:(1,2)}] at (-2,6) {};
        
	\coordinate (B7) at (0,-2);
	\coordinate (B6) at (2,-2);
	\coordinate (B5) at (4,-2);
 	\coordinate (B4) at (6,-2);
	\coordinate (B3) at (8,-2);
	\coordinate (B2) at (10,-2);
	\coordinate (B1) at (12,-2);

	\node[draw,circle,inner sep=5pt,fill,blue!60] at (B7) {};
	\node[draw,circle,inner sep=5pt,fill,blue!60] at (B6) {};
	\node[draw,circle,inner sep=5pt,fill,blue!60] at (B5) {};
	\node[draw,circle,inner sep=5pt,fill,blue!60] at (B4) {};
	\node[draw,circle,inner sep=5pt,fill,blue!60] at (B3) {};
	\node[draw,circle,inner sep=5pt,fill,blue!60] at (B2) {};
	\node[draw,circle,inner sep=8pt,fill,blue] at (B1) {};

	\node[coordinate, pin = {225:$b_1$=(8,9)}] at (B1) {};
	\node[coordinate, pin = {225:$b_2$=(7,9)}] at (B2) {};
	\node[coordinate, pin = {225:$b_3$=(6,9)}] at (B3) {};
	\node[coordinate, pin = {225:$b_4$=(5,9)}] at (B4) {};
	\node[coordinate, pin = {225:$b_5$=(4,9)}] at (B5) {};
	\node[coordinate, pin = {225:$b_6$=(3,9)}] at (B6) {};
	\node[coordinate, pin = {225:$b_7$=(2,9)}] at (B7) {};

	\coordinate (Shift1) at (2,-2);
	\coordinate (Shift2) at (2,-4);
    \draw[style=help lines,dashed] (-2,-2) grid[step=2cm] (12,8);

    \foreach \x in {-1,0,...,6}{
      \foreach \y in {-1,0,...,4}{
        \node[draw,circle,inner sep=2pt,fill,] at (2*\x,2*\y) {};

      }
    }

    \foreach \x in {-1,-0,...,6}{
      \foreach \y in {0,1,...,4}{
        \node[draw,circle,inner sep=2pt,fill,] at (2*\x,2*\y) {};
	 \draw [-stealth,black,line width=2pt] (2*\x,2*\y)
        -- ($(2*\x,2*\y)+(0,-2)$)
	node[pos=0.4,left,black, outer sep=-3pt] {\LARGE 1};
	
      }
    }


	\foreach \x in {-1,0,...,6}{
	
		\FPeval{\z}{clip(5-0)};
		 \draw [->>,black,line width=2pt] (2*\x,0) -- ($(2*\x,0)+(Shift1)$) node [black,midway,fill=white] {\LARGE $v_{\z}$};
	}

\foreach \x in {-1,0,...,6}{
	\foreach \y in {1,...,4}{
		\node[draw,circle,inner sep=2pt,fill,] at (2*\x,2*\y) {};
			\FPeval{\z}{clip(5-\y)};
			\draw [->>,black,line width=2pt] (2*\x,2*\y) -- ($(2*\x,2*\y)+(Shift2)$) node [black,midway,fill=white] {\LARGE $v_{\z}$ };
	}
}


\draw [->>,black!50!green,line width=4pt] (A1) -- (10,0);
\draw [->>,black!50!green,line width=4pt] (10,0) -- (B2);

\draw [->>,black!50!green,line width=4pt] (A2) -- (8,4);
\draw [->>,black!50!green,line width=4pt] (8,4) -- (B3);

\draw [->>,green!80,line width=4pt] (A3) -- (8,4);
\draw [->>,green!80,line width=4pt] (8,4) -- (10,0);
\draw [->>,green!80,line width=4pt] (10,0) -- (B1);

\draw [->>,blue!50!green!90,line width=4pt] (A4) -- (4,2);
\draw [->>,blue!50!green!90,line width=4pt] (4,2) -- (B4);

\draw [->>,green!90,line width=4pt] (A5) -- (2,6);
\draw [->>,green!90,line width=4pt] (2,6) -- (4,2);
\draw [->>,green!90,line width=4pt] (4,2) -- (B5);

\draw [->>,blue!40!green,line width=4pt] (A6) -- (0,2);
\draw [->>,blue!40!green,line width=4pt] (0,2) -- (B6);

\draw [->>,blue!20!green,line width=4pt] (A7) -- (-2,0);
\draw [->>,blue!20!green,line width=4pt] (-2,0) -- (B7);


\node[ draw,circle,inner sep=8pt,fill,orange] at (4,2) {};
\node[] at (4,2) {\LARGE{$z$}};

\node[ draw,circle,inner sep=8pt,fill,orange!70] at (8,4) {};
\node[] at (8,4) {\LARGE{$z'$}};

\node[ draw,circle,inner sep=8pt,fill,orange!40] at (10,0) {};
\node[] at (10,0) {\LARGE{$z''$}};

\end{tikzpicture}
\caption{The grid $\Gamma_{U_5}$ and multipath $\delta_{z}(\vec{\rho})$.}
\label{Pmultipath1switched}

\end{figure}
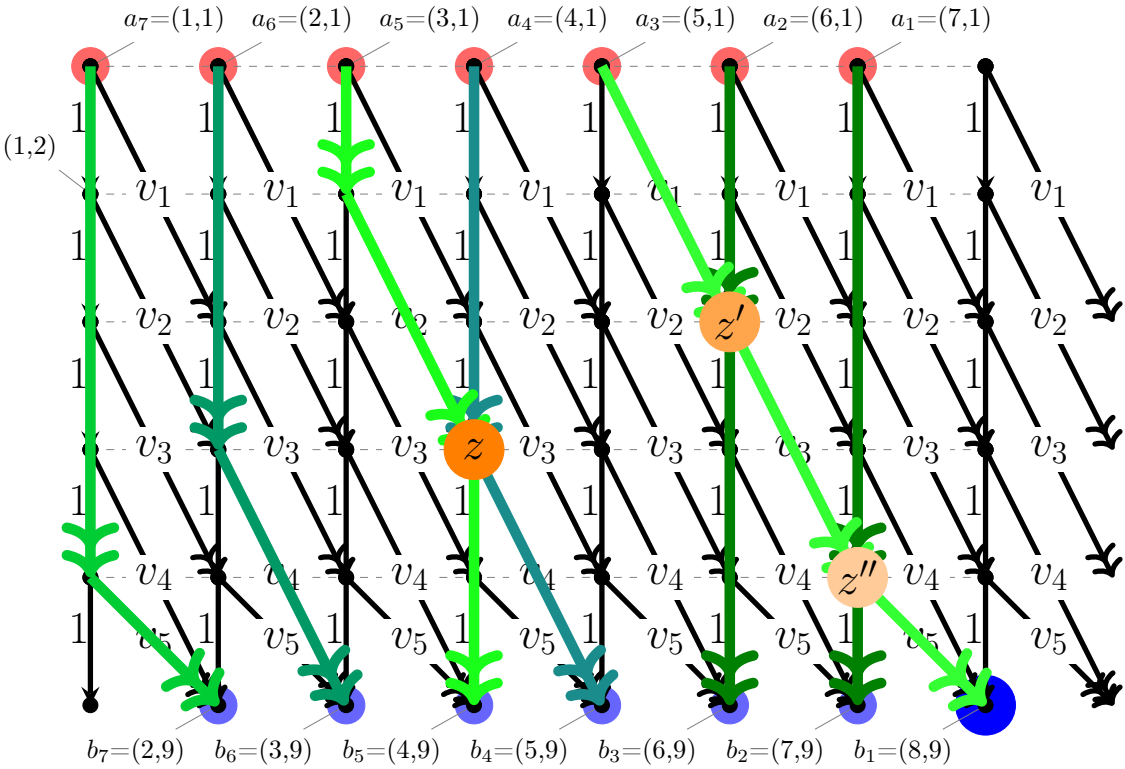

\newpage

We denote by $J_k$ the following set of weights vectors, which describe multipaths like on the Picture \ref{Pmultipath1preproper1}:
$$J_k=\{(1^{l-1},v_{i_1}\cdot ...\cdot v_{i_l},v_{i_{l+1}},...v_{i_k})\ | 1\leq l\leq k;\ v_{i_j}\prec v_{i_{j+1}}, \text{ if } 1\leq j\leq l;\ v_{i_1}\nsucc v_{i_{l+1}};\ v_{i_j}\nsucc v_{i_{j+1}}, \text{ if } l< j\leq k\}.$$
Denote by $\overline{J}_k$ the corresponding set of mutipaths, which are uniquely defined by the vectors of its weights.

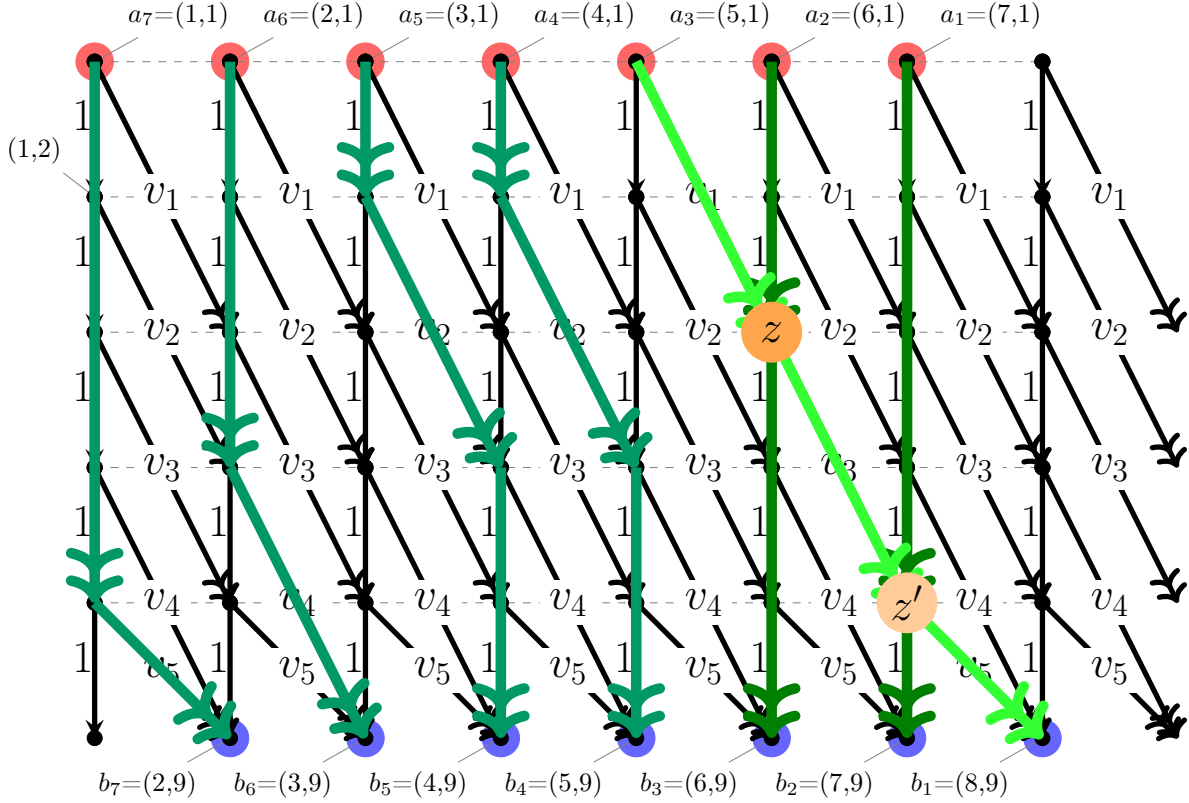
\begin{figure}[H]
\hspace*{-0.5cm} 
\begin{tikzpicture}[scale=0.9]

    \coordinate (Origin)   at (0,0);

 	\coordinate (A7) at (-2,8);
	\coordinate (A6) at (0,8);
	\coordinate (A5) at (2,8);
	\coordinate (A4) at (4,8);
 	\coordinate (A3) at (6,8);
	\coordinate (A2) at (8,8);
	\coordinate (A1) at (10,8);

	\node[draw,circle,inner sep=5pt,fill,red!60] at (A7) {};
	\node[draw,circle,inner sep=5pt,fill,red!60] at (A6) {};
	\node[draw,circle,inner sep=5pt,fill,red!60] at (A5) {};
	\node[draw,circle,inner sep=5pt,fill,red!60] at (A4) {};
	\node[draw,circle,inner sep=5pt,fill,red!60] at (A3) {};
	\node[draw,circle,inner sep=5pt,fill,red!60] at (A2) {};
	\node[draw,circle,inner sep=5pt,fill,red!60] at (A1) {};
	
	\node[coordinate, pin = {45:$a_1$=(7,1)}] at (A1) {};
	\node[coordinate, pin = {45:$a_2$=(6,1)}] at (A2) {};
	\node[coordinate, pin = {45:$a_3$=(5,1)}] at (A3) {};
	\node[coordinate, pin = {45:$a_4$=(4,1)}] at (A4) {};
	\node[coordinate, pin = {45:$a_5$=(3,1)}] at (A5) {};
	\node[coordinate, pin = {45:$a_6$=(2,1)}] at (A6) {};
	\node[coordinate, pin = {45:$a_7$=(1,1)}] at (A7) {};

	\node[coordinate, pin = {135:(1,2)}] at (-2,6) {};
        
	\coordinate (B7) at (0,-2);
	\coordinate (B6) at (2,-2);
	\coordinate (B5) at (4,-2);
 	\coordinate (B4) at (6,-2);
	\coordinate (B3) at (8,-2);
	\coordinate (B2) at (10,-2);
	\coordinate (B1) at (12,-2);

	\node[draw,circle,inner sep=5pt,fill,blue!60] at (B7) {};
	\node[draw,circle,inner sep=5pt,fill,blue!60] at (B6) {};
	\node[draw,circle,inner sep=5pt,fill,blue!60] at (B5) {};
	\node[draw,circle,inner sep=5pt,fill,blue!60] at (B4) {};
	\node[draw,circle,inner sep=5pt,fill,blue!60] at (B3) {};
	\node[draw,circle,inner sep=5pt,fill,blue!60] at (B2) {};
	\node[draw,circle,inner sep=5pt,fill,blue!60] at (B1) {};

	\node[coordinate, pin = {225:$b_1$=(8,9)}] at (B1) {};
	\node[coordinate, pin = {225:$b_2$=(7,9)}] at (B2) {};
	\node[coordinate, pin = {225:$b_3$=(6,9)}] at (B3) {};
	\node[coordinate, pin = {225:$b_4$=(5,9)}] at (B4) {};
	\node[coordinate, pin = {225:$b_5$=(4,9)}] at (B5) {};
	\node[coordinate, pin = {225:$b_6$=(3,9)}] at (B6) {};
	\node[coordinate, pin = {225:$b_7$=(2,9)}] at (B7) {};

	\coordinate (Shift1) at (2,-2);
	\coordinate (Shift2) at (2,-4);

    \draw[style=help lines,dashed] (-2,-2) grid[step=2cm] (12,8);

    \foreach \x in {-1,0,...,6}{
      \foreach \y in {-1,0,...,4}{
        \node[draw,circle,inner sep=2pt,fill,] at (2*\x,2*\y) {};

      }
    }

    \foreach \x in {-1,-0,...,6}{
      \foreach \y in {0,1,...,4}{
        \node[draw,circle,inner sep=2pt,fill,] at (2*\x,2*\y) {};
	 \draw [-stealth,black,line width=2pt] (2*\x,2*\y)
        -- ($(2*\x,2*\y)+(0,-2)$)
	node[pos=0.4,left,black, outer sep=-3pt] {\LARGE 1};
	
      }
    }


	\foreach \x in {-1,0,...,6}{
	
		\FPeval{\z}{clip(5-0)};
		 \draw [->>,black,line width=2pt] (2*\x,0) -- ($(2*\x,0)+(Shift1)$) node [black,midway,fill=white] {\LARGE $v_{\z}$};
	}

\foreach \x in {-1,0,...,6}{
	\foreach \y in {1,...,4}{
		\node[draw,circle,inner sep=2pt,fill,] at (2*\x,2*\y) {};
			\FPeval{\z}{clip(5-\y)};
			\draw [->>,black,line width=2pt] (2*\x,2*\y) -- ($(2*\x,2*\y)+(Shift2)$) node [black,midway,fill=white] {\LARGE $v_{\z}$ };
	}
}

\draw [->>,black!50!green,line width=4pt] (A1) -- (10,0);
\draw [->>,black!50!green,line width=4pt] (10,0) -- (B2);

\draw [->>,black!50!green,line width=4pt] (A2) -- (8,4);
\draw [->>,black!50!green,line width=4pt] (8,4) -- (B3);

\draw [->>,green!80,line width=4pt] (A3) -- (8,4);
\draw [->>,green!80,line width=4pt] (8,4) -- (10,0);
\draw [->>,green!80,line width=4pt] (10,0) -- (B1);

\draw [->>,blue!40!green,line width=4pt] (A4) -- (4,6);
\draw [->>,blue!40!green,line width=4pt] (4,6) -- (6,2);
\draw [->>,blue!40!green,line width=4pt] (6,2) -- (B4);

\draw [->>,blue!40!green,line width=4pt] (A5) -- (2,6);
\draw [->>,blue!40!green,line width=4pt] (2,6) -- (4,2);
\draw [->>,blue!40!green,line width=4pt] (4,2) -- (B5);

\draw [->>,blue!40!green,line width=4pt] (A6) -- (0,2);
\draw [->>,blue!40!green,line width=4pt] (0,2) -- (B6);

\draw [->>,blue!40!green,line width=4pt] (A7) -- (-2,0);
\draw [->>,blue!40!green,line width=4pt] (-2,0) -- (B7);


\node[ draw,circle,inner sep=8pt,fill,orange!70] at (8,4) {};
\node[] at (8,4) {\LARGE{$z$}};

\node[ draw,circle,inner sep=8pt,fill,orange!40] at (10,0) {};
\node[] at (10,0) {\LARGE{$z'$}};

\end{tikzpicture}
\caption{The grid $\Gamma_{U_5}$ and a multipath without easy intersection points. }
\label{Pmultipath1preproper1}

\end{figure}

Let $\vec{\rho}\in(\overline{\Omega}_k\setminus \overline{I}_k)\setminus\overline{P}_k$.
\begin{itemize}
\item If $\vec{\rho}$ is intersecting, then the absolute value of the difference between the multipliers of
$\vec{\rho}$ and $\delta_{z}(\vec{\rho}\ )$ in the sum ($\ref{sumJ}$) is equal to $1$:
$$ |\sigma_{\vec{\rho}}^{-1}(1)-\sigma({\delta_{z}(\vec{\rho}}))^{-1}(1)|=1,$$
because if $\rho_l$ goes to $b_1$, then $z$ could only be obtained as an intersection of $\rho_l$ and $\rho_{l-1}$ or $\rho_{l+1}$. Hence, since $\delta_{z}(\vec{\rho})\in(\overline{\Omega}_k\setminus \overline{I}_k)\setminus\overline{P}_k$, we make a switch at $z$ and eliminate one of the switched multipaths (from $(\overline{\Omega}_k\setminus \overline{I}_k)\setminus\overline{P}_k\setminus\overline{J}_k$) and the multiplier of the multipath with longer intersecting path (from $\overline{J}_k$) in the sum~($\ref{sumJ}$). 

\item If $\vec{\rho}\in(\overline{\Omega}_k\setminus \overline{I}_k)\setminus\overline{P}_k$ is non-intersecting, then its multiplier is also equal to $1$. Denote the set of such multipaths by $\overline{L}_k$:
$$\overline{L}_k=\{ \vec{\rho}\in(\overline{\Omega}_k\setminus \overline{I}_k)\setminus\overline{P}_k| \  \vec{\rho} \text{ is non-intersecting}  \}.$$
Then, $$L_k=\{ (w_1,...,w_k)|\ w_i\nsucc w_{i+1}; \exists m,\text{ s.t. } w_m\succ \max\limits_{1\leq q<m}{w_q} \}.$$
\end{itemize}

Hence, we can rewrite the sum ($\ref{sumJ}$) in the following way:
\begin{equation}\label{sumK}
\sum\limits_{\vec{\rho}\in\overline{\Omega}_k\setminus \overline{I}_k\setminus\overline{P}_k}\mathrm{sign}(\sigma_{\vec{\rho}})\cdot\sigma_{\vec{\rho}}^{-1}(1)\prod\limits_{i=1}^k w(\rho_i)=\sum\limits_{\vec{\rho}\in\overline{J}_k\sqcup\overline{L}_k}\mathrm{sign}(\sigma_{\vec{\rho}})\prod\limits_{i=1}^k w(\rho_i).
\end{equation}

To eliminate the sum~(\ref{sumJ}), we construct a sign-reversing involution on $J_k\sqcup L_k$. Let
$$A_k=\{ (1^{l-1},w_1\cdot ...\cdot w_l,w_{l+1},...w_k)\in J_k|\ l>1 \text{ and } w_j\nsucc \max\limits_{1\leq q<j}{w_q} \text{ for every } 1\leq j\leq k  \}.$$
$$B_k=\{ (1^{l-1},w_1\cdot ...\cdot w_l,w_{l+1},...w_k)\in J_k\sqcup L_k|\ \exists m,\text{ s.t. } w_m\succ \max\limits_{1\leq q<m}{w_q} \}.$$
Then, we have $$J_k\sqcup L_k=A_k\sqcup B_k.$$
Next, we construct a sign-reversing bijection between $A_k$ and $B_k$.

First, we define map $\chi: A_k\to B_k$. If  $\vec{v}=(1^{l-1},w_1\cdot ...\cdot w_l,w_{l+1},...w_k)\in A_k$, then let $$m=\max\{j\leq n|\ w_m\succ w_j \text{ for } j<m\}.$$ We set $$\chi(\vec{v}\ )= (1^{l-1},w_1\cdot ...\cdot w_l\cdot w_m,w_{l+1},...,w_{m-1},w_{m+1},...,w_k)\in B_k.$$ 
Note that $\chi$ changes the sign $\vec{v}$ by increasing its $l$-th weight by 1.
Second, if $$\vec{u}=(1^{l-1},w_1\cdot ...\cdot w_l,w_{l+1},...w_k)\in A_k,$$ then we set $$m'=\max\{j\leq k|\ w_l\succ w_j\},$$ and define $$\psi(\vec{u}\ ) = (1^{l-1},w_1\cdot ...\cdot w_{l-1},w_{l+1},...,w_{m'},w_l,w_{m'+1},...w_k)\in B_k.$$
Note that $\psi=\chi^{-1}.$
For instance, the multipath from Picture \ref{Pmultipath1preproper1}, which belongs to $A_k$, is transformed to the below mutipath (Picture \ref{Pmultipath1preproper2}) under the action of $\psi$, and vice versa Picture \ref{Pmultipath1preproper1} can be obtained from Picture  \ref{Pmultipath1preproper2} applying direct map $\chi$:

\begin{figure}[H]
\hspace*{-0.5cm} 
\begin{tikzpicture}[scale=0.9]

    \coordinate (Origin)   at (0,0);

 	\coordinate (A7) at (-2,8);
	\coordinate (A6) at (0,8);
	\coordinate (A5) at (2,8);
	\coordinate (A4) at (4,8);
 	\coordinate (A3) at (6,8);
	\coordinate (A2) at (8,8);
	\coordinate (A1) at (10,8);

	\node[draw,circle,inner sep=5pt,fill,red!60] at (A7) {};
	\node[draw,circle,inner sep=5pt,fill,red!60] at (A6) {};
	\node[draw,circle,inner sep=5pt,fill,red!60] at (A5) {};
	\node[draw,circle,inner sep=5pt,fill,red!60] at (A4) {};
	\node[draw,circle,inner sep=5pt,fill,red!60] at (A3) {};
	\node[draw,circle,inner sep=5pt,fill,red!60] at (A2) {};
	\node[draw,circle,inner sep=5pt,fill,red!60] at (A1) {};
	
	\node[coordinate, pin = {45:$a_1$=(7,1)}] at (A1) {};
	\node[coordinate, pin = {45:$a_2$=(6,1)}] at (A2) {};
	\node[coordinate, pin = {45:$a_3$=(5,1)}] at (A3) {};
	\node[coordinate, pin = {45:$a_4$=(4,1)}] at (A4) {};
	\node[coordinate, pin = {45:$a_5$=(3,1)}] at (A5) {};
	\node[coordinate, pin = {45:$a_6$=(2,1)}] at (A6) {};
	\node[coordinate, pin = {45:$a_7$=(1,1)}] at (A7) {};

	\node[coordinate, pin = {135:(1,2)}] at (-2,6) {};
        
	\coordinate (B7) at (0,-2);
	\coordinate (B6) at (2,-2);
	\coordinate (B5) at (4,-2);
 	\coordinate (B4) at (6,-2);
	\coordinate (B3) at (8,-2);
	\coordinate (B2) at (10,-2);
	\coordinate (B1) at (12,-2);

	\node[draw,circle,inner sep=5pt,fill,blue!60] at (B7) {};
	\node[draw,circle,inner sep=5pt,fill,blue!60] at (B6) {};
	\node[draw,circle,inner sep=5pt,fill,blue!60] at (B5) {};
	\node[draw,circle,inner sep=5pt,fill,blue!60] at (B4) {};
	\node[draw,circle,inner sep=5pt,fill,blue!60] at (B3) {};
	\node[draw,circle,inner sep=5pt,fill,blue!60] at (B2) {};
	\node[draw,circle,inner sep=5pt,fill,blue!60] at (B1) {};

	\node[coordinate, pin = {225:$b_1$=(8,9)}] at (B1) {};
	\node[coordinate, pin = {225:$b_2$=(7,9)}] at (B2) {};
	\node[coordinate, pin = {225:$b_3$=(6,9)}] at (B3) {};
	\node[coordinate, pin = {225:$b_4$=(5,9)}] at (B4) {};
	\node[coordinate, pin = {225:$b_5$=(4,9)}] at (B5) {};
	\node[coordinate, pin = {225:$b_6$=(3,9)}] at (B6) {};
	\node[coordinate, pin = {225:$b_7$=(2,9)}] at (B7) {};

	\coordinate (Shift1) at (2,-2);
	\coordinate (Shift2) at (2,-4);

      \draw[style=help lines,dashed] (-2,-2) grid[step=2cm] (12,8);

    \foreach \x in {-1,0,...,6}{
      \foreach \y in {-1,0,...,4}{
        \node[draw,circle,inner sep=2pt,fill,] at (2*\x,2*\y) {};

      }
    }

    \foreach \x in {-1,-0,...,6}{
      \foreach \y in {0,1,...,4}{
        \node[draw,circle,inner sep=2pt,fill,] at (2*\x,2*\y) {};
	 \draw [-stealth,black,line width=2pt] (2*\x,2*\y)
        -- ($(2*\x,2*\y)+(0,-2)$)
	node[pos=0.4,left,black, outer sep=-3pt] {\LARGE 1};
	
      }
    }


	\foreach \x in {-1,0,...,6}{
	
		\FPeval{\z}{clip(5-0)};
		 \draw [->>,black,line width=2pt] (2*\x,0) -- ($(2*\x,0)+(Shift1)$) node [black,midway,fill=white] {\LARGE $v_{\z}$};
	}

\foreach \x in {-1,0,...,6}{
	\foreach \y in {1,...,4}{
		\node[draw,circle,inner sep=2pt,fill,] at (2*\x,2*\y) {};
			\FPeval{\z}{clip(5-\y)};
			\draw [->>,black,line width=2pt] (2*\x,2*\y) -- ($(2*\x,2*\y)+(Shift2)$) node [black,midway,fill=white] {\LARGE $v_{\z}$ };
	}
}


\draw [->>,black!50!green,line width=4pt] (A1) -- (10,4);
\draw [->>,black!50!green,line width=4pt] (10,4) -- (B2);

\draw [->>,green!80,line width=4pt] (A2) -- (10,4);
\draw [->>,green!80,line width=4pt] (10,4) -- (12,0);
\draw [->>,black!50!green,line width=4pt] (12,0) -- (B1);

\draw [->>,blue!40!green,line width=4pt] (A3) -- (6,6);
\draw [->>,blue!40!green,line width=4pt] (6,6) -- (8,2);
\draw [->>,blue!40!green,line width=4pt] (8,2) -- (B3);

\draw [->>,blue!40!green,line width=4pt] (A4) -- (4,6);
\draw [->>,blue!40!green,line width=4pt] (4,6) -- (6,2);
\draw [->>,blue!40!green,line width=4pt] (6,2) -- (B4);

\draw [->>,black!50!green,line width=4pt] (A5) -- (2,0);
\draw [->>,green!80,line width=4pt] (2,0) -- (B5);

\draw [->>,blue!40!green,line width=4pt] (A6) -- (0,2);
\draw [->>,blue!40!green,line width=4pt] (0,2) -- (B6);

\draw [->>,blue!40!green,line width=4pt] (A7) -- (-2,0);
\draw [->>,blue!40!green,line width=4pt] (-2,0) -- (B7);


\node[ draw,circle,inner sep=8pt,fill,orange!70] at (10,4) {};
\node[] at (10,4) {\LARGE{$z$}};

\node[ draw,circle,inner sep=8pt,fill,black!40] at (12,0) {};
\node[] at (12,0) {\LARGE{$z'$}};

\node[ draw,circle,inner sep=8pt,fill,black!40] at (2,0) {};
\node[] at (2,0) {\LARGE{$z'$}};

\end{tikzpicture}
\caption{Image of the multipath from Picture \ref{Pmultipath1preproper1} under the action of $\chi$. }
\label{Pmultipath1preproper2}

\end{figure}
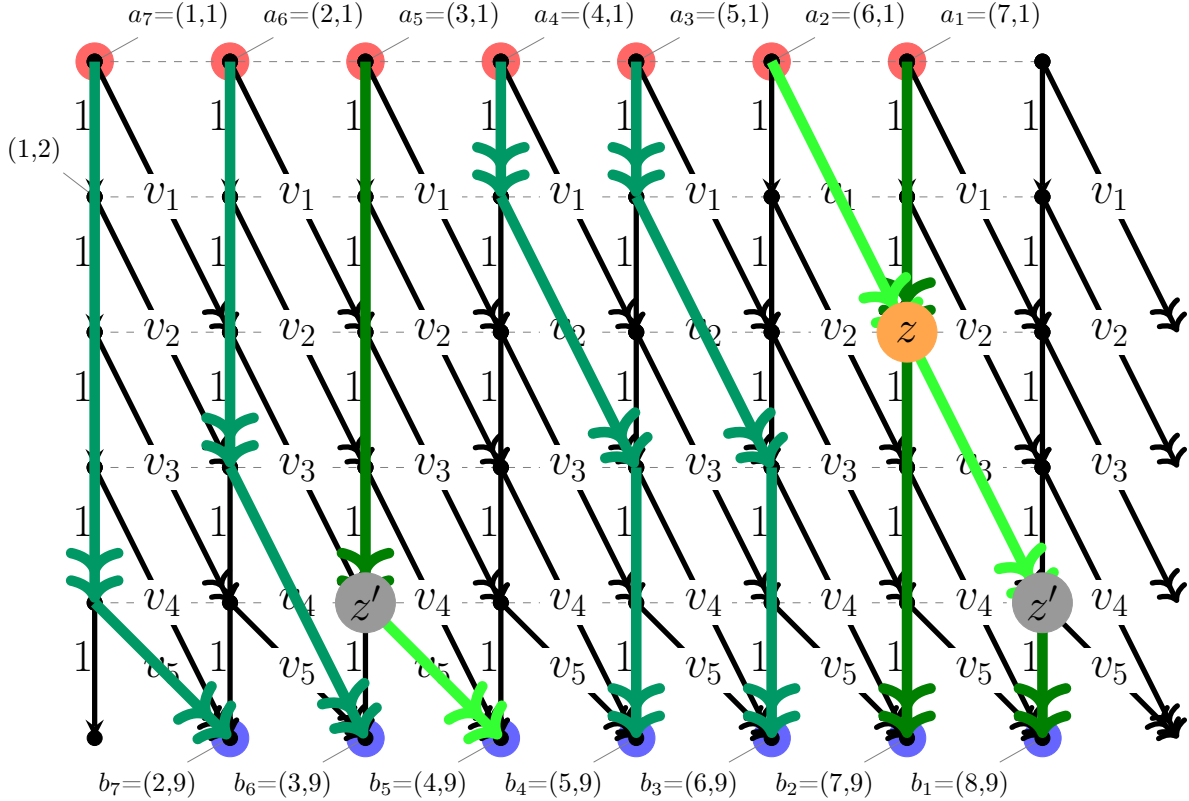

Hence, among  the sums (\ref{sumJ}), (\ref{sumI}) and (\ref{sumW}), only the latter is non-zero, we have only the set of corrects left:$$p_k^U=\sum\limits_{ \substack{(\rho_1,...,\rho_k):A\rightarrow B \\ (w(\rho_1),...,w(\rho_k))\in P^U_k}}\prod\limits_{i=1}^n w(\rho_i).$$

\end{proof}

\newcommand{\ind}{l}
This result will play a major role in our future work. The construction of a correct sequence allows to work with $m^U$ functions by expanding them in terms of $p^U$ functions. For instance, using the following relation
$$m^U_{\ind,1}=p^U_\ind\cdot p^U_1 - p^U_{\ind+1},$$
it is easy to prove the following

\begin{thm}\label{Thn1}
Let $$M_{\ind,1}^U = \{(\vs\ |z)\in P^U_{\ind}\times P^U_1| \; z\succ\vs\vee z\prec w_\ind\},$$

then $$m^U_{\ind,1}=\sum\limits_{(\vs;z)\in M^U_{\ind,1}}w_1\cdot...\cdot w_\ind\cdot z.$$
\end{thm}
\begin{remark}
According to Remark~\ref{c_m}, this implies $c_{n-1,1}(U)\geq 0$.
\end{remark}

Here, we omit the proof. This approach will be used in the next article, where positivity of some $e$-coefficients for $(3+1)$-free posets will be demonstrated.


\end{document}